\documentclass{amsart}
\usepackage{amsmath}
\usepackage{amssymb}
\usepackage{amsthm}
\usepackage{enumerate}
\usepackage[dvips]{graphicx}
\theoremstyle{definition}
\newtheorem{definition}{Definition}[section]
\theoremstyle{plain}
\newtheorem{lemma}[definition]{Lemma}
\newtheorem{theorem}[definition]{Theorem}
\newtheorem{proposition}[definition]{Proposition}
\newtheorem{corollary}[definition]{Corollary}
\theoremstyle{remark}
\newtheorem{remark}[definition]{Remark}

\newtheorem*{conjecture}{Conjecture}

\newcommand{\myint}{\operatorname{int}}

\makeatletter
\@namedef{subjclassname@2020}{
  \textup{2020} Mathematics Subject Classification}
\makeatother

\begin{document}

\title[Definable functions]{Functions definable in definably complete uniformly locally o-minimal structure of the second kind}
\author[M. Fujita]{Masato Fujita}
\address{Department of Liberal Arts,
Japan Coast Guard Academy,
5-1 Wakaba-cho, Kure, Hiroshima 737-8512, Japan}
\email{fujita.masato.p34@kyoto-u.jp}

\begin{abstract}
We investigate continuous functions definable in a definably complete uniformly locally o-minimal expansion of the second kind of a densely linearly ordered abelian group (DCULOAS structure).

We prove a variant of the Arzela-Ascoli theorem for uniformly continuous definable functions and the following assertion:
Consider the parameterized function $f:C \times P \rightarrow M$ which is equi-continuous with respect to $P$.
The projection image of the set at which $f$ is discontinuous to the parameter space $P$ is of dimension smaller than $\dim P$ when $C$ is closed and bounded.

In addition, we demonstrate that an archimedean DCULOAS structure which enjoys definable Tietze extension property is o-minimal.
In the appendix, we show that an o-minimal expansion of an ordered group is not semi-bounded if and only if it enjoys definable Tietze extension property.
\end{abstract}



\maketitle

\section{Introduction}\label{sec:intro}
An o-minimal structure enjoys tame properties such as monotonicity and definable cell decomposition \cite{vdD,KPS,PS}. 
Toffalori and Vozoris first introduced locally o-minimal structures in \cite{TV}.
Roughly speaking, a locally o-minimal structure is defined by simply localizing the definition of an o-minimal structure.
See their paper \cite{TV} for the precise definition of locally o-minimal structures.
In spite of its similarity to the definition of o-minimal structures, a locally o-minimal structure does not enjoy the localized counterparts such as local monotonicity theorem and local definable cell decomposition theorem.

A uniformly locally o-minimal structure of the second kind was first introduced in \cite{Fuji} as a structure which enjoys a local monotonicity theorem \cite[Theorem 3.2]{Fuji} and a local definable cell decomposition theorem \cite[Theorem 4.2]{Fuji} with the extra cost of definably completeness \cite{M}.
If it is also an expansion of a densely linearly ordered abelian group, it also enjoys more tame properties such as good dimension theory and a decomposition into special submanifolds \cite{Fuji, Fuji2, Fuji3}.
We consider a definably complete uniformly locally o-minimal expansion of the second kind of a densely linearly ordered abelian group in this paper.
We call it a DCULOAS structure for short.

We investigate functions definable in a DCULOAS structure in this paper.
The good dimension theory enables us to investigate definable continuous functions.
For instance, a definable continuous function on a definable closed bounded set is uniformly continuous.
We also demonstrate a variant of the Arzela-Ascoli theorem for definable functions.
Consider the parameterized function $f:C \times P \rightarrow M$ which is equi-continuous with respect to $P$.
One of our main theorems is that the projection image of the set at which $f$ is discontinuous to the parameter space $P$ is of dimension smaller than $\dim P$ when $C$ is closed and bounded.

Functions definable in a DCULOAS structure enjoy several tame properties as above.
Definable Tietze extension theorem is a convenient tool for topological studies of structures such as \cite{T}.
It is available in a definably complete expansion of an ordered field \cite[Lemma 6.6]{AF}.
Unfortunately, a uniformly locally o-minimal expansion of the second kind of an ordered field is o-minimal \cite[Proposition 2.1]{Fuji}.
The author's concern is whether a DCULOAS structure enjoys definable Tietze extension property.
We obtain a negative partial result on this conjecture.
An archimedean DCULOAS structure which enjoys definable Tietze extension property is o-minimal. 
We also demonstrate that an o-minimal expansion of an ordered group is not semi-bounded if and only if it enjoys definable Tietze extension property in the appendix.

We introduce the terms and notations used in this paper.
The term `definable' means `definable in the given structure with parameters' in this paper.
For a linearly ordered structure $\mathcal M=(M,<,\ldots)$, an open interval is a definable set of the form $\{x \in R\;|\; a < x < b\}$ for some $a,b \in M$.
It is denoted by $]a,b[$ in this paper.
We define a closed interval similarly. 
It is denoted by $[a,b]$.
An open box in $M^n$ is the direct product of $n$ open intervals.
When the structure $\mathcal M$ is an expansion of an abelian group.
The notation $M_{>r}$ denotes the set $\{x \in M\;|\;x>r\}$ for any $r \in M$.
We set $|x|:=\max_{1 \leq i \leq n}|x_i|$ for any vector $x = (x_1, \ldots, x_n) \in M^n$.
The function $|x-y|$ defines a distance in $M^n$.
Let $A$ be a subset of a topological space.
The notations $\myint(A)$ and $\overline{A}$ denote the interior and the closure of the set $A$, respectively.

This paper is organized as follows:
We first review the previous works in Section \ref{sec:review}.
The dimension theory of sets definable in a definably complete uniformly locally o-minimal structure of the second kind is reviewed in this section.
Definable choice lemma for a DCULOAS structure is necessary for our study.
Section \ref{sec:definable_choice} is devoted to the lemma and its corollaries.
We prove the main theorems other than the monotonicity theorem and the theorem on Tietze extension in Section \ref{sec:property} using the assertions in the previous sections.
We consider definable Tietze extension property in Section \ref{sec:tietze}.
In Section \ref{sec:appendix}, we investigate definable Tietze extension property in o-minimal structures.

\section{Review of previous works}\label{sec:review}
We first review the dimension theory of sets definable in a definably complete uniformly locally o-minimal structure of the second kind.
A definably complete uniformly locally o-minimal structure of the second kind admits local definable cell decomposition by \cite[Theorem 4.2]{Fuji}.
We review the definitions of cells and local definable cell decomposition in \cite[Definition 4.1]{Fuji}.

\begin{definition}[Definable cell decomposition]
Consider a densely linearly ordered structure $\mathcal M=(M,<,\ldots)$.
Let $(i_1, \ldots, i_n)$ be a sequence of zeros and ones of length $n$.
\textit{$(i_1, \ldots, i_n)$-cells} are definable subsets of $M^n$ defined inductively as follows:
\begin{itemize}
\item A $(0)$-cell is a point in $M$ and a $(1)$-cell is an open interval in $M$.
\item An $(i_1,\ldots,i_n,0)$-cell is the graph of a continuous definable function defined on an $(i_1,\ldots,i_n)$-cell.
An $(i_1,\ldots,i_n,1)$-cell is a definable set of the form $\{(x,y) \in C \times M\;|\; f(x)<y<g(x)\}$, where $C$ is an $(i_1,\ldots,i_n)$-cell and $f$ and $g$ are definable continuous functions defined on $C$ with $f<g$.
\end{itemize}
A \textit{cell} is an $(i_1, \ldots, i_n)$-cell for some sequence $(i_1, \ldots, i_n)$ of zeros and ones.
An \textit{open cell} is a $(1,1, \ldots, 1)$-cell.

We inductively define a \textit{definable cell decomposition} of an open box $B \subset M^n$.
For $n=1$, a definable cell decomposition of $B$ is a partition $B=\bigcup_{i=1}^m C_i$ into finite cells.
For $n>1$, a definable cell decomposition of $B$ is a partition $B=\bigcup_{i=1}^m C_i$ into finite cells such that $\pi(B)=\bigcup_{i=1}^m \pi(C_i)$ is a definable cell decomposition of $\pi(B)$, where $\pi:M^n \rightarrow M^{n-1}$ is the projection forgetting the last coordinate.
Consider a finite family $\{A_\lambda\}_{\lambda \in \Lambda}$ of definable subsets of $B$.
A \textit{definable cell decomposition of $B$ partitioning $\{A_\lambda\}_{\lambda \in \Lambda}$} is a definable cell decomposition of $B$ such that the definable sets $A_{\lambda}$ are unions of cells for all $\lambda \in \Lambda$. 
\end{definition}

When a locally o-minimal structure admits local definable cell decomposition, we can define the dimension of definable sets and it enjoys several good properties.
Three equivalent definitions of dimension are given in \cite[Definition 5.1, Corollary 5.3]{Fuji}.
We only review the assertions on dimension in \cite{Fuji, Fuji2} which are necessary for this study.

\begin{proposition}\label{prop:pre}
Let $\mathcal M=(M,<,\ldots)$ be a definably complete uniformly locally o-minimal structure of the second kind.
The following assertions hold true:
\begin{enumerate}
\item[(1)] Let $X \subset Y$ be definable sets.
Then, the inequality $\dim(X) \leq \dim(Y)$ holds true.
\item[(2)] Let $X$ be a subset of $M^n$. The set $X$ is of dimension $n$ if and only if $X$ has a nonempty interior.
\item[(3)] A definable subset of $M^n$ is of dimension $ \geq d$ if it contains an $(i_1, \ldots, i_n)$-cell with $\sum_{j=1}^n i_j \geq d$.
\item[(4)] Let $X $ be a nonempty definable subset of $M^n$.
There exists a point $x \in X$ such that $\dim (X \cap U)=\dim X$ for any open box $U$ containing the point $x$.
\item[(5)] Let $X$ and $Y$ be definable subsets of $M^n$.
We have $$\dim(X \cup Y)=\max\{\dim(X),\dim(Y)\}\text{.}$$
\item[(6)] Let $X$ be a definable set.
The frontier $\overline{X} \setminus X$ is of dimension smaller than $\dim X$.
\item[(7)] Assume that $\mathcal M$ is a DCULOAS structure.
Let $f:X \rightarrow R$ be a definable function.
The set of the points at which $f$ is discontinuous is of dimension smaller than $\dim(X)$.
\end{enumerate}
\end{proposition}
\begin{proof}
(1) \cite[Lemma 5.1]{Fuji}; 
(2) through (4) \cite[Corollary 5.3]{Fuji};
(5) \cite[Corollary 5.4(ii)]{Fuji}; 
(6) \cite[Theorem 5.6]{Fuji}; 
(7) \cite[Corollary 1.2]{Fuji2}.
\end{proof}

We also need the following lemma on the dimension of a definable subset in $M$.

\begin{lemma}\label{lem:discrete}
Let $\mathcal M=(M,<,\ldots)$ be a definably complete uniformly locally o-minimal structure of the second kind.
Let $X$ be a definable subset of $M$.
It is of dimension zero if and only if it is discrete and closed.
If $X$ is of dimension zero and bounded below, we have $\inf X \in X$.
\end{lemma}
\begin{proof}
The first assertion follows from \cite[Lemma 5.2, Corollary 5.3]{Fuji}.
Let $X$ be a definable subset of $M$ of dimension zero.
Since $\mathcal M$ is definably complete, the infimum $\inf X$ is well-defined.
It is finite because $X$ is bounded below.
Since $X$ is closed, we have $\inf X \in X$.
\end{proof}

The following technical lemma is used in Section \ref{sec:property}.

\begin{lemma}\label{lem:ato}
Let $\mathcal M=(M,<,\ldots)$ be a definably complete uniformly locally o-minimal structure of the second kind.
Consider definable subsets $X_1$ and $X_2$ of $M^m$ and $M^n$, respectively.
Let $\pi:M^{m+n} \rightarrow M^m$ be the projection onto the first $m$ coordinates.
Let $W$ be an open box in $M^{m+n}$.
Consider a definable cell decomposition of $W$ partitioning $W \cap (X_1 \times X_2)$.
Assume that a cell $E$ contained in $X_1 \times X_2$ satisfies the following conditions:
\begin{enumerate}
\item[(a)] $\dim \pi(E)=\dim X_1$;
\item[(b)] Any cell $E'$ such that $\dim E'> \dim E$, $\pi(E')=\pi(E)$ and $\overline{E'} \cap E \neq \emptyset$ is not contained in $X_1 \times X_2$.
\end{enumerate}
Then there exist a point $x \in E$ and an open box $U$ in $M^{m+n}$ containing the point $x$ such that $(X_1 \times X_2) \cap U$ is contained in $E$.
\end{lemma}
\begin{proof}
Let $E$ be an $(i_1,\ldots,i_{m+n})$-cell.
Consider the family $\mathcal D$ of the cells in $\pi(W)$ contained in $X_1$ which is not $\pi(E)$. 
Set $\widetilde{D}=\pi(E) \setminus \bigcup_{D \in \mathcal D} \overline{D}$.
We have $\dim \widetilde{D}=\dim X_1$ by Proposition \ref{prop:pre}(5) because $\dim \overline{D} \cap \pi(E) \leq \dim \overline{D} \setminus D<\dim D \leq\dim X_1$ for all $D \in\mathcal D$ by Proposition \ref{prop:pre}(1), (6).
Let $D'$ be a cell contained in $\widetilde{D}$ with $\dim D'=\dim X_1$, the intersection $\pi^{-1}(D') \cap E$ is obviously a cell of dimension $\dim D'+i_{m+1}+\cdots+i_{m+n}=\dim E$ by the definition of dimension.
The definable set $F=\pi^{-1}(\widetilde{D}) \cap E$ is of dimension $\dim E$ by Proposition \ref{prop:pre}(3).

Consider the family $\mathcal E$ of the cells $E'$ in $W$ contained in $X_1 \times X_2$ such that $\pi(E)=\pi(E')$ and $\overline{E'} \cap E \neq \emptyset$.
We have $\dim E' \leq \dim E$ by the assumption.
Set $G=F \setminus \bigcup_{E' \in \mathcal E} \overline{E'}$.
We have $\dim G=\dim E$ for the same reason as above.
In particular, the definable set $G$ is not an empty set.
Take a point $x \in G$ and a sufficiently small open box $U$ containing $x$.
It is obvious that $D \cap \pi(U)=\emptyset$ for all $D \in \mathcal D$ and $E' \cap U=\emptyset$ for all $E' \in \mathcal E$ by the definition.
Take an arbitrary cell $E''$ contained in $X_1 \times X_2$ which is not a member of $\mathcal E$ and $\pi(E'')=\pi(E)$.
We have $\overline{E''} \cap E=\emptyset$.
Hence, we may assume that $E'' \cap U=\emptyset$ for such cells $E''$ shrinking $U$ if necessary.
We have demonstrated that $(X_1 \times X_2) \cap U$ is contained in $E$ because $U$ has an empty intersection with the cells contained in $X_1 \times X_2$ other than $E$.
\end{proof}

The following monotonicity theorem holds true.

\begin{theorem}[Monotonicity theorem]\label{thm:mono}
Let $\mathcal M=(M,<,+,0,\ldots)$ be a DCULOAS structure.
Let $I$ be an interval and $f:I \rightarrow M$ be a definable function.
There exists a partition $I=X_d \cup X_c \cup X_+ \cup X_-$ of $I$ into definable sets satisfying the following conditions:
\begin{enumerate}
\item[(1)] the definable set $X_d$ is discrete and closed;
\item[(2)] the definable set $X_c$ is open and $f$ is locally constant on $X_c$;
\item[(3)] the definable set $X_+$ is open and $f$ is locally strictly increasing and continuous on $X_+$;
\item[(4)] the definable set $X_-$ is open and $f$ is locally strictly decreasing and continuous on $X_-$.
\end{enumerate}
\end{theorem}
\begin{proof}
Immediate from \cite[Theorem 2.10, Proposition 2.12]{Fuji3}.
\end{proof}

We also need the following lemma in \cite{Fuji}.
\begin{lemma}\label{lem:dc_mono}
Let $\mathcal M=(M,<,\ldots)$ be a definably complete local o-minimal structure.
 A locally strictly monotone definable function defined on an open interval is strictly monotone.
\end{lemma}
\begin{proof}
\cite[Proposition 3.1]{Fuji}
\end{proof}

The following corollary guarantees the existence of the limit.

\begin{corollary}\label{cor:limit}
Let $\mathcal M=(M,<,+,0,\ldots)$ be a DCULOAS structure.
Let $s>0$ and $f:]0,s[ \rightarrow M^n$ be a bounded definable map.
There exists a unique point $x \in M^n$ satisfying the following condition:
$$
\forall \varepsilon >0, \exists \delta>0, \forall t, \ 0<t<\delta \Rightarrow |x-f(t)| < \varepsilon \text{.}
$$ 
The notation $\lim_{t \to +0}f(t)$ denotes the point $x$.
\end{corollary}
\begin{proof}
We first reduce to the case in which $n=1$.
Assume that the corollary holds true for $n=1$.
Let $\pi_i$ be the projection onto the $i$-th coordinate for all $1 \leq i \leq n$.
Apply the corollary to the composition $\pi_i \circ f$.
Set $x_i=\lim_{t \to +0}\pi_i \circ f(t)$ for all $1 \leq i \leq n$.
It is obvious that $x=(x_1, \ldots, x_n)$ is the unique point satisfying the condition in the corollary.
We have succeeded in reducing to the case in which $n=1$.

Set $I=]0,s[$.
Applying Theorem \ref{thm:mono} to $f$, we get a partition $I=X_d \cup X_c \cup X_+ \cup X_-$ into definable sets satisfying the conditions in Theorem \ref{thm:mono}.
Since $X_d$ is discrete and closed, we have $\inf X_d \in X_d$ by Lemma \ref{lem:discrete}.
Shrinking the interval $I$ if necessary, we may assume that $X_d$ is an empty set.
Since the interval $I$ is definably connected by \cite[Proposition 1.4]{M}, we have $I=X_c$, $I=X_+$ or $I=X_-$.
We only consider the case in which $I=X_-$.
We can prove the corollary similarly in the other cases.

The function $f$ is strictly decreasing by Lemma \ref{lem:dc_mono} because $I=X_-$.
Set $x=\inf_{0<t<s}f(t)$, which exists because $f$ is bounded.
It is obvious the point $x$ satisfies the required condition because $f$ is strictly decreasing.
Let $x'$ be another point satisfying the condition.
We fix an arbitrary $\varepsilon >0$.
There exists $\delta>0$ with $|x-f(t)|<\varepsilon$ whenever $0<t<\delta$.
There exists $\delta'>0$ with $|x'-f(t)|<\varepsilon$ whenever $0<t<\delta'$.
Set $\delta''=\min \{\delta,\delta'\}$.
We have $|x-x'| \leq |x-f(t)|+|x'-f(t)| < 2\varepsilon$ whenever $0<t<\delta''$.
We get $x=x'$ because $\varepsilon$ is an arbitrary positive element.
\end{proof}

\section{Definable choice}\label{sec:definable_choice}

We review the following definable choice lemma and its applications.

\begin{lemma}[Definable choice]\label{lem:definable_choice}
Consider a definably complete expansion of a densely linearly ordered abelian group  $\mathcal M=(M,<,+,0,\ldots)$.
Let $X$ be a definable subset of $M^{m+n}$.
The notation $\pi:M^{m+n} \rightarrow M^n$ denotes the projection onto the last $n$ coordinates.
There exists a definable map $\varphi:\pi(X) \rightarrow X$ such that the composition $\pi \circ \varphi$ is the identity map on $\pi(X)$.
\end{lemma}
\begin{proof}
Let $c$ be a positive element in $M$.
For any nonempty definable set $J$ in $M$, we define $e(J)$ as follows:

We set $e(J)=0$ if $0 \in J$.
Otherwise, if the intersection $]0 ,+\infty[ \cap J$ is not empty, set $a=\inf\{x>0\;|\;x \in J\}$ and $b=\sup \{x >0\;|\; ]a,x[ \subset J\}$.
Set $e(J)=\frac{a+b}{2}$ if $b < \infty$ and $e(J)=a+c$ otherwise.
The former is well-defined because $(M,+)$ is a divisible group by \cite[Proposition 2.2]{M}.
In the remaining case, the intersection $]-\infty,0[ \cap J$ is not empty.
Set $a=\sup\{x<0\;|\;x \in J\}$ and $b=\inf \{x <0\;|\; ]x,a[ \subset J\}$.
Set $e(J)=\frac{a+b}{2}$ if $b > -\infty$ and $e(J)=a-c$ otherwise.
It is obvious that $e(J) \in J$ in all the cases.

We demonstrate the lemma by induction on $m$.
For any $x \in \pi(X)$, we set $X_x=\{y \in M^m\;|\; (y,x) \in X\}$.
When $m=1$, the definable function $\varphi:\pi(X) \rightarrow X$ given by $\varphi(x)=e(X_x)$ satisfies the requirement.
When $m>1$, consider the projections $\pi_1:M^{m+n} \rightarrow M^{m+n-1}$ and $\pi_2:M^{m+n-1} \rightarrow M^n$.
The former forgets the first coordinate and the latter is the projection onto the last $n$ coordinates.
It is obvious that $\pi=\pi_2 \circ \pi_1$.
There exist definable maps $\varphi_1:\pi_1(X) \rightarrow X$ and $\varphi_2:\pi(X) \rightarrow \pi_1(X)$ such that $\pi_i \circ \varphi_i$ are the identity maps on the domains of definition of $\varphi_i$ for $i=1,2$.
The composition $\varphi=\varphi_1 \circ \varphi_2$ satisfies the requirement.
\end{proof}

The following curve selection lemma is worth to be mentioned.

\begin{corollary}\label{cor:curve_selection}
Consider a DCULOAS structure $\mathcal M=(M,<,+,0,\ldots)$.
Let $X$ be a definable subset of $M^n$ which is not closed.
Take a point $a \in \overline{X} \setminus X$.
There exist a small positive $\varepsilon$ and a definable continuous map $\gamma:]0,\varepsilon[ \rightarrow X$ such that $\lim_{t \to +0}\gamma(t)=a$.
\end{corollary}
\begin{proof}
Let $\pi:M^{n+1} \rightarrow M$ be the projection onto the last coordinate.
Set $Y=\{(x,t) \in X \times M\;|\; |a-x|=t\}$.
Since $\mathcal M$ is locally o-minimal, the intersection $]-\delta, \delta[ \cap \pi(Y)$ is a finite union of points and open intervals for a sufficiently small $\delta>0$.
Since the point $a$ belongs to the closure of $X$, the intersection $]-\delta, \delta[ \cap \pi(Y)$ contains an open interval of the form $]0,\varepsilon[$ for some $\varepsilon>0$.
There exists a definable map $\gamma:]0,\varepsilon[ \rightarrow X$ with $(\gamma(t),t) \in Y$ for all $0 < t < \varepsilon$ by Lemma \ref{lem:definable_choice}.
The set $D$ of points at which the definable function $\gamma$ is discontinuous is of dimension zero by Proposition \ref{prop:pre}(7).
We have $\inf D \in D$ by Lemma \ref{lem:discrete}.
In particular, we get $\inf D > 0$.
Taking a smaller $\varepsilon>0$ if necessary, we may assume that $\gamma$ is continuous.
The equality $\lim_{t \to +0}\gamma(t)=a$ is obvious by the definition of $\gamma$.
\end{proof}

The following two lemmas are used in the subsequent section.
They can be proved by using the definable choice lemma.

\begin{lemma}\label{lem:proj_dim}
Consider a DCULOAS structure $\mathcal M=(M,<,+,0,\ldots)$.
Let $C$ and $P$ be definable subsets of $M^m$ and $M^n$, respectively.
Let $X$ be a definable subset of $C \times P$.
The notation $\pi:M^{m+n} \rightarrow M^n$ denotes the projection onto the last $n$ coordinates.
Assume that $\dim \pi(X)=\dim P$.
Then there exists a point $(c,p) \in X$ such that $\dim \pi(X \cap W) = \dim P$ for all open boxes $W$ in $M^{m+n}$ containing the point $(c,p)$.
\end{lemma}
\begin{proof}
We can find a definable map $\tau:\pi(X) \rightarrow X$ such that the composition $\pi \circ \tau$ is the identity map on $\pi(X)$ by Lemma \ref{lem:definable_choice}.
Let $D$ be the closure of the set of points at which $\tau$ is discontinuous.
We have $\dim D < \dim \pi(X)=\dim P$ by Proposition \ref{prop:pre}(5), (6), (7).
Set $E=\pi(X) \setminus D$.
We obtain $\dim E=\dim P$ by Proposition \ref{prop:pre}(5).
Therefore there exists a point $p \in E$ with $\dim (E \cap U)=\dim P$ for all open box $U$ in $M^n$ containing the point $p$ by Proposition \ref{prop:pre}(4).
Set $(c,p)=\tau(p)$.

We demonstrate that the point $(c,p)$ satisfies the condition in the lemma.
Take an arbitrary sufficiently small open box $W$ in $M^{m+n}$ containing the point $(c,p)$.
We may assume that $D \cap \pi(W) = \emptyset$ because $p \not\in D$ and $D$ is closed.
Since $\tau$ is continuous on $E$, the set $\tau^{-1}(W)=\pi(\tau(E) \cap W)$ is open in $E$.
There exists an open box $U$ in $R^n$ such that $p \in U$ and $E \cap U \subset \pi(\tau(E) \cap W)$.
Shrinking $U$ if necessary, we may assume that $U$ is contained in $\pi(W)$.
We have $\dim P=\dim E \cap U$ by the definition of the point $p$.
We then get $\dim P=\dim E \cap U \leq \dim \pi(\tau(E) \cap W) \leq \dim \pi(X \cap W) \leq \dim P$ by Proposition \ref{prop:pre}(1).
We have demonstrated the lemma.
\end{proof}

\begin{lemma}\label{lem:main}
Consider a DCULOAS structure $\mathcal M=(M,<,+,0,\ldots)$.
Let $C$ be a definable closed and bounded subset of $R^m$.
Let $\varphi, \psi: C \rightarrow M_{>0}$ be two definable functions.
Assume that the following condition is satisfied:
$$
\forall x \in C, \exists \delta>0, \forall x' \in C,\ |x'-x| < \delta \Rightarrow \varphi(x') \geq \psi(x) \text{.}
$$
Then we have $\inf \varphi(C)>0$.
\end{lemma}
\begin{proof}
Set $l=\inf\varphi(C) \geq 0$, which exists by the definably completeness of $\mathcal M$.
We have only to show that $l>0$.
Since $\mathcal M$ is locally o-minimal, we have $l \in \varphi(C)$ or there exists $u \in M$ with $l < u$ and $]l,u[ \subset \varphi(C)$.
It is obvious that $l>0$ in the former case.
We consider the latter case in the rest of the proof.

Let $\Gamma$ be the graph of the function $\varphi$.
Let $\pi_1:M^{m+1} \rightarrow M^m$ and $\pi_2:M^{m+1} \rightarrow M$ be the projections onto the first $m$ coordinates and onto the last coordinate, respectively.
We can take a definable map $\eta:]l,u[ \rightarrow \Gamma$ such that the composition $\pi_2 \circ \eta$ is the identity map on $]l,u[$ by Lemma \ref{lem:definable_choice}.
Note that the map $\eta$ is bounded because the domain of definition $C$ of $\varphi$ is bounded and the interval $]l,u[$ is bounded.
Since the set of points at which $\eta$ is discontinuous is at most of dimension zero by Proposition \ref{prop:pre}(7), we may assume that $\eta$ is continuous by taking a smaller $u$ if necessary.

Set $z=\lim_{t \to l+}\eta(t)$, which uniquely exists by Corollary \ref{cor:limit}.
We have $\pi_2(z)=l$ by the definition of $\eta$.
Set $c=\pi_1(z)$.
It belongs to $C$ because $C$ is bounded and closed.
For any $t>l$ sufficiently close to $l$, $\pi_1(\eta(t)) \in C$ is close to the point $c$.
We have $\pi_2(\eta(t)) \geq \psi(c)$ for such $t$ by the assumption.
We finally obtain $l=\lim_{t \to l+}\pi_2(\eta(t))) \geq \psi(c)>0$.
\end{proof}

\section{Properties of definable functions}\label{sec:property}
We investigate the properties of functions definable in a DCULOAS structure.
\begin{definition}
Consider an expansion of a densely linearly ordered abelian group $\mathcal M=(M,<,+,0,\ldots)$.
Let $C$ and $P$ be definable sets.
Let $f: C \times P \rightarrow M$ be a definable function.
The function $f$ is \textit{equi-continuous} with respect to $P$ if the following condition is satisfied:
$$
\forall \varepsilon>0, \ \forall x \in C, \ \exists \delta >0, \ \forall p \in P, \ \forall x' \in C,\ \  |x-x'|< \delta \Rightarrow |f(x,p)-f(x',p)|< \varepsilon\text{.}
$$
The function $f$ is \textit{uniformly equi-continuous} with respect to $P$ if the following condition is satisfied:
$$
\forall \varepsilon>0, \ \exists \delta >0, \ \forall p \in P, \ \forall x, x' \in C,\ \  |x-x'|< \delta \Rightarrow |f(x,p)-f(x',p)|< \varepsilon\text{.}
$$

The function $f$ is \textit{pointwise bounded} with respect to $P$ if the following condition is satisfied:
$$
\forall x \in C,\  \exists N>0, \  \forall p \in P,\ \   |f(x,p)|<N\text{.}
$$
\end{definition}

\begin{proposition}\label{prop:equi-cont}
Consider a DCULOAS structure $\mathcal M=(M,<,+,0,\ldots)$.
Let $C$ and $P$ be definable sets.
Let $f: C \times P \rightarrow M$ be a definable function.
Assume that $C$ is closed and bounded.
Then $f$ is equi-continuous with respect to $P$ if and only if it is uniformly equi-continuous with respect to $P$.
\end{proposition}
\begin{proof}
A uniformly equi-continuous definable function is always equi-continuous.
We prove the opposite implication.

Take a positive $c \in M$.
Consider the definable function $\varphi:C \times M_{>0} \rightarrow M_{>0}$ given by 
$$ \varphi(x,\varepsilon)=\sup\{0<\delta<c\;|\;\forall p \in P, \ \forall x' \in C,\  |x-x'|< \delta \Rightarrow |f(x,p)-f(x',p)|< \varepsilon\}\text{.}$$
Since $f$ is equi-continuous with respect to $P$, we have $\varphi(x,\varepsilon)>0$ for all $x \in C$ and $\varepsilon>0$.
Fix arbitrary $x \in C$ and $\varepsilon>0$.
We also fix an arbitrary point $x' \in C$ with $|x'-x|<\frac{1}{2}\varphi(x,\frac{\varepsilon}{2})$.
We have $|f(x',p)-f(x,p)|<\frac{\varepsilon}{2}$ by the definition of $\varphi$.

For all $y \in C$ with $|x'-y|<\frac{1}{2}\varphi(x,\frac{\varepsilon}{2})$, we have $|x-y| \leq |x-x'|+|x'-y| < \varphi(x,\frac{\varepsilon}{2})$.
We get $|f(y,p)-f(x,p)|<\frac{\varepsilon}{2}$ by the definition of $\varphi$.
We finally obtain $|f(y,p)-f(x',p)| \leq |f(x',p)-f(x,p)|+|f(y,p)-f(x,p)|<\varepsilon$.
It means that $\varphi(x',\varepsilon) \geq \frac{1}{2}\varphi(x,\frac{\varepsilon}{2})$ whenever $|x'-x|<\frac{1}{2}\varphi(x,\frac{\varepsilon}{2})$.
Apply Lemma \ref{lem:main} to the definable functions $\varphi(\cdot,\varepsilon)$ and $\frac{1}{2}\varphi(\cdot,\frac{\varepsilon}{2})$ for a fixed $\varepsilon>0$.
We have $\inf \varphi(C,\varepsilon)>0$.

For any $\varepsilon>0$, set $\delta=\inf \varphi(C,\varepsilon)$.
For any $p \in P$ and $x,x' \in C$, we have $|f(x,p)-f(x',p)|< \varepsilon$ whenever $|x-x'|< \delta$ by the definition of $\varphi$.
It means that $f$ is uniformly equi-continuous.
\end{proof}

It is well known that a continuous function defined on a compact set is uniformly continuous.
The following corollary claims that a similar assertion holds true for a definable function defined on a definable closed bounded set.
\begin{corollary}\label{cor:uniform}
Consider a DCULOAS structure $\mathcal M=(M,<,+,0,\ldots)$.
Let $C$ be a definable closed and bounded set.
A definable continuous function $f: C \rightarrow M$ is uniformly continuous.
\end{corollary}
\begin{proof}
Let $P$ be a singleton.
Apply Proposition \ref{prop:equi-cont} to the function $g:C \times P \rightarrow M$ defined by $g(x,p)=f(x)$.
\end{proof}

We define a definable family of functions and investigate its properties.
Equi-continuity, convergence and uniform convergence are defined for sequences of functions in classical analysis.
We consider similar notions for a definable family of functions.
\begin{definition}
Consider an expansion of a densely linearly ordered abelian group $\mathcal M=(M,<,+,0,\ldots)$.
Let $C$ be a definable set and $s$ be a positive element in $M$.
A family $\{f_t:C \rightarrow M\}_{0<t<s}$ of functions with the parameter variable $t$ is a \textit{definable family of functions} if there exists a definable function $F:C \times ]0,s[ \rightarrow M$ such that $f_t(x)=F(x,t)$ for all $x \in C$ and $0<t<s$.
We call it a \textit{definable family of continuous functions} if every function $f_t$ is continuous.

Consider a definable family of functions $\{f_t:C \rightarrow M\}_{0<t<s}$.
Set $I=]0,s[$.
The map $F:C \times I \rightarrow M$ given by $F(x,t) =f_t(x)$ is a definable function by the definition.
The family is a \textit{definable family of equi-continuous functions} if $F$ is equi-continuous with respect to $I$.
It is a \textit{definable family of pointwise bounded functions} if $F$ is pointwise bounded with respect to $I$.

A definable family of functions $\{f_t:C \rightarrow M\}_{0<t<s}$ is \textit{pointwise convergent} if for all positive $\varepsilon>0$ and for all $x \in C$, there exists $s'>0$ such that $|f_t(x)-f_{t'}(x)|<\varepsilon$ for all $t,t' \in ]0,s'[$.
\end{definition}

The following lemma is proved following a typical argument in classical analysis.
\begin{lemma}\label{lem:pointwise}
Consider an expansion of a densely linearly ordered abelian group $\mathcal M=(M,<,+,0,\ldots)$.
Let $C$ be a definable set and $s$ be a positive element in $M$.
Consider a pointwise convergent definable family of functions $\{f_t:C \rightarrow M\}_{0<t<s}$.
For any $x \in C$, there exists $s'>0$ such that the set $\{f_t(x)\;|\;0<t<s'\}$ is bounded.
\end{lemma}
\begin{proof}
Fix $x \in C$.
Take a positive $\varepsilon >0$.
There exists $s'>0$ such that$|f_t(x)-f_{t'}(x)|<\varepsilon$ for all $t,t' \in ]0,s'[$.
Fix $u \in ]0,s'[$.
For any $t \in ]0,s'[$, we have $|f_t(x)| \leq |f_u(x)|+|f_u(x)-f_t(x)| < |f_u(x)|+ \varepsilon$.
It means that the set $\{f_t(x)\;|\;0<t<s'\}$ is bounded.
\end{proof}

We also get the following converse when $\mathcal M$ is a DCULOAS structure.
\begin{lemma}\label{lem:pointwise2}
Let $\mathcal M=(M,<,+,0,\ldots)$ be a DCULOAS structure.
Let $C$ be a definable set and $s$ be a positive element in $M$.
A definable family of pointwise bounded functions $\{f_t:C \rightarrow M\}_{0<t<s}$ is pointwise convergent.
\end{lemma}
\begin{proof}
Fix $x \in C$.
Set $I=]0,s[$.
Consider the definable function $g:I \rightarrow M$ given by $g(t)=f_t(x)$.
It is bounded.
There exists a limit $y=\lim_{t \to +0}g(t)$ by Corollary \ref{cor:limit}.

Take a positive $\varepsilon >0$.
There exists $s'>0$ such that$|y-g(t)|<\varepsilon/2$ for all $t\in ]0,s'[$.
We have $|f_t(x)-f_{t'}(x)| \leq |f_t(x)-y|+|y-f_{t'}(x)|<\varepsilon$ whenever $t, t'\in ]0,s'[$.
It means that the family $\{f_t:C \rightarrow M\}_{0<t<s}$ is pointwise convergent.
\end{proof}

We define the limit of a pointwise convergent definable family of functions.
\begin{definition}
Consider a DCULOAS structure $\mathcal M=(M,<,+,0,\ldots)$.
Let $C$ be a definable set and $s$ be a positive element in $M$.
Consider a pointwise convergent definable family of functions $\{f_t:C \rightarrow M\}_{0<t<s}$.
For any $x \in C$, consider the function $g_x:]0,s[ \rightarrow M$ given by $g_x(t)=f_t(x)$.
Taking a smaller $s>0$ if necessary, we may assume that $g_x$ is bounded by Lemma \ref{lem:pointwise}.
There exists a unique limit $\lim_{t \to +0} g_x(t)$ exists by Corollary \ref{cor:limit}.
The \textit{limit} $\lim_{t \to +0}f_t:C \rightarrow M$ of the family $\{f_t:C \rightarrow M\}_{0<t<s}$ is defined by $(\lim_{t \to +0}f_t)(x)=\lim_{t \to +0}g_x(t)$.
\end{definition}

\begin{definition}
Consider an expansion of a densely linearly ordered abelian group $\mathcal M=(M,<,+,0,\ldots)$.
Let $C$ be a definable set and $s$ be a positive element in $M$.
A definable family of functions $\{f_t:C \rightarrow M\}_{0<t<s}$ is \textit{uniformly convergent} if for all positive $\varepsilon>0$, there exists $s'>0$ such that $|f_t(x)-f_{t'}(x)|<\varepsilon$ for all $x \in C$ and $t,t' \in ]0,s'[$.
\end{definition}

The following proposition and its proof is almost the same as the counterparts in classical analysis.
\begin{proposition}\label{prop:uniform_convergent}
Consider a DCULOAS structure $\mathcal M=(M,<,+,0,\ldots)$.
Let $C$ be a definable set and $s$ be a positive element in $M$.
Consider a uniformly convergent definable family of continuous functions $\{f_t:C \rightarrow M\}_{0<t<s}$.
The limit $\lim_{t \to +0}f_t:C \rightarrow M$ is continuous.
\end{proposition}
\begin{proof}
Fix arbitrary $\varepsilon>0$ and $x \in C$.
Since the family is uniformly convergent, we may assume that $|f_t(x')-f_{t'}(x')|<\frac{\varepsilon}{5}$ for all $x' \in C$ and $t,t' \in ]0,s[$ by taking a smaller $s>0$ if necessary.
Fix $t_0$ with $0<t_0<s$.
There exists $\delta>0$ such that $|f_{t_0}(x')-f_{t_0}(x)|<\frac{\varepsilon}{5}$ whenever $|x-x'|<\delta$ because $f_{t_0}$ is continuous.
Fix a point $x' \in C$ with $|x-x'|<\delta$.
We can take $t_1,t_2 \in ]0,s[$ with $|(\lim_{t \to +0}f_t)(x)-f_{t_1}(x)|<\frac{\varepsilon}{5}$ and $|(\lim_{t \to +0}f_t)(x')-f_{t_2}(x')|<\frac{\varepsilon}{5}$ by the definition of the limit $\lim_{t \to +0}f_t$.
We finally have $|(\lim_{t \to +0}f_t)(x')-(\lim_{t \to +0}f_t)(x)| \leq |(\lim_{t \to +0}f_t)(x')-f_{t_2}(x')|+|f_{t_2}(x')-f_{t_0}(x')|+|f_{t_0}(x')-f_{t_0}(x)|+|f_{t_0}(x)-f_{t_1}(x)|+|f_{t_1}(x)-(\lim_{t \to +0}f_t)(x)|<\varepsilon$.
We have proven that $\lim_{t \to +0}f_t$ is continuous.
\end{proof}

The following Arzela-Ascoli-type theorem is a main theorem of this paper.

\begin{theorem}\label{thm:ascoli}
Consider a DCULOAS structure $\mathcal M=(M,<,+,0,\ldots)$.
Let $C$ be a definable closed and bounded set.
A pointwise convergent definable family of equi-continuous functions $\{f_t:C \rightarrow M\}_{0<t<s}$ is uniformly convergent. 
\end{theorem}
\begin{proof}
Set $I=]0,s[$.
Consider the map $F:C \times I \rightarrow M$ given by $F(x,t) =f_t(x)$.
It is an equi-continuous definable function with respect to $I$ by the definition.
Set $g=\lim_{t \to +0} f_t$.
It is well-defined because the family is pointwise convergent.

Take $c>0$.
Consider the definable function $\varphi:C \times M_{>0} \rightarrow M_{>0}$ given by
$$\varphi(x,\varepsilon)=\sup\{0<\delta<c\;|\; \forall t,t' \in ]0,\delta[,\ |F(x,t)-F(x,t')|<\varepsilon\}\text{.}$$
 We first show that it is well-defined.
 Fix $x \in C$ and $\varepsilon>0$.
 There exists $\delta>0$ such that $|F(x,u)-g(x)|<\frac{\varepsilon}{2}$ for all $u \in ]0,\delta[$ by the definition of $g$.
For any $t,t' \in  ]0,\delta[$, we have $|F(x,t)-F(x,t')| \leq |F(x,t)-g(x)|+|g(x)-F(x,t')|<\varepsilon$.
The definable set $\{0<\delta<c\;|\; \forall t,t' \in ]0,\delta[,\ |F(x,t)-F(x,t')|<\varepsilon\}$ is not empty and the function $\varphi$ is well-defined.

We fix $x \in C$ and $\varepsilon>0$ again.
Since $F$ is equi-continuous with respect to $I$, there exists $\delta'>0$ such that 
$$\forall t \in ]0,s[,\ \forall x' \in C, \ |x-x'|< \delta' \Rightarrow |F(x,t)-F(x',t)|<\frac{\varepsilon}{3}\text{.}$$
Fix arbitrary $x' \in C$ with $|x-x'|< \delta'$.
For any $t,t' \in ]0,\varphi(x, \frac{\varepsilon}{3})[$, we have $|F(x,t)-F(x,t')|<\frac{\varepsilon}{3}$ by the definition of $\varphi$.
We finally get $$|F(x',t)-F(x',t')| \leq |F(x',t)-F(x,t)|+|F(x,t)-F(x,t')|+|F(x,t')-F(x',t')|<\varepsilon$$
whenever $t,t' \in ]0,\varphi(x, \frac{\varepsilon}{3})[$.
It means that $\varphi(x',\varepsilon) \geq \varphi(x,\frac{\varepsilon}{3})$.
Apply Lemma \ref{lem:main} to the definable functions $\varphi(\cdot,\varepsilon)$ and $\varphi(\cdot,\frac{\varepsilon}{3})$ for a fixed $\varepsilon>0$.
We have $\inf \varphi(C,\varepsilon)>0$ for all $\varepsilon>0$.

Fix $\varepsilon>0$.
Set $\delta= \inf \varphi(C,\varepsilon)>0$.
We have $|f_t(x)-f_{t'}(x)|=|F(t,x)-F(t',x)|<\varepsilon$ for all $x \in C$ and $t,t' \in ]0,\delta[$.
It means that the family $\{f_t:C \rightarrow M\}_{0<t<s}$ is uniformly convergent. 
\end{proof}

The above theorem together with the curve selection lemma yields the following corollary:
\begin{corollary}\label{cor:ascoli}
Consider a DCULOAS structure $\mathcal M=(M,<,+,0,\ldots)$.
Let $C$ and $P$ be definable sets.
Assume that $C$ is closed and bounded.
Let $f:C \times P \rightarrow M$ be a definable function which is equi-continuous and pointwise bounded with respect to $P$.
Take $p \in \overline{P}$.
There exists a definable continuous curve $\gamma:]0,\varepsilon[ \rightarrow P$ such that $\lim_{t \to +0}\gamma(t)=p$ and the definable family of functions $\{g_t:C \rightarrow M\}_{0<t<\varepsilon}$ defined by $g_t(x)=f(x,\gamma(t))$ is uniformly convergent.
\end{corollary}
\begin{proof}
The corollary follows from Corollary \ref{cor:curve_selection}, Lemma \ref{lem:pointwise2} and Theorem \ref{thm:ascoli}.
\end{proof}

Consider a parameterized function $f:C \times P \rightarrow M$ which is equi-continuous with respect to $P$.
The following theorem claims that the projection image of the set at which $f$ is discontinuous onto the parameter space $P$ is of dimension smaller than $\dim P$ when $C$ is closed and bounded.
\begin{theorem}\label{thm:discont}
Consider a DCULOAS structure $\mathcal M=(M,<,+,0,\ldots)$.
Let $C$ be a definable closed and bounded set and $P$ be a definable set.
Let $\pi:C \times P \rightarrow P$ be the projection.
Consider a definable function $f:C \times P \rightarrow M$ which is equi-continuous with respect to $P$.
Set $D= \{(x,q) \in C \times P\;|\; f \text{ is discontinuous at }(x,q)\}$.
We have $\dim \pi(D) < \dim P$.
\end{theorem}
\begin{proof}
Let $C$ and $P$ be definable subsets of $M^m$ and $M^n$, respectively.
We first consider the set
$$
S=\{(x,p) \in D\;|\; \exists U \subset M^m:\text{open box with }x \in U \text{ and }C \cap U = D_p \cap U\}\text{,}$$
where the notation $D_p$ denotes the fiber $\{x \in C\;|\; (x,p) \in D\}$.
We first demonstrate that $\dim \pi(S) < \dim P$.

Assume the contrary.
There exists a point $(c,p) \in S$ such that $\dim \pi(S \cap W)=\dim P$ for all open box $W$ in $M^{m+n}$ containing the point $(c,p)$ by Lemma \ref{lem:proj_dim}.
Fix a sufficiently small open box $W$ containing the point $(c,p)$.
Let $\tau:C \times P \rightarrow P \times C$ be the map defined by $\tau(x,p)=(p,x)$ and $\pi':M^{m+n} \rightarrow M^n$ be the projection onto the first $n$ coordinates.
Shrinking $W$ if necessary, there exists a definable cell decomposition of $W'=\tau(W)$ partitioning the definable sets $W' \cap (P \times C)$ and $S'=\tau(S \cap W)$ by \cite[Theorem 4.2]{Fuji}.
There exists a cell $E$ contained in $S'$ with $\dim\pi'(E)=\dim P$ by the assumption.
Let $E_{\max}$ be a cell of the maximum dimension among such cells.

Let $E'$ be a cell such that $\pi'(E')=\pi'(E_{\max})$, $\dim E' > \dim E_{\max}$ and $\overline{E'} \cap E_{\max} \neq \emptyset$.
We show that $E' \cap (P \times C)$ is an empty set.
Assume the contrary.
The cell $E'$ is contained in $P \times C$ because it is a cell of the cell decomposition partitioning the set $W' \cap (P \times C)$.
Take $(p',c') \in \overline{E'} \cap E_{\max}$.
We obviously have $(p',c') \in P \times C$ because $E_{\max} \subset S' \subset P \times C$.
Since $(p',c')$ is an element of $S'$, there exists an open box $U'$ in $M^m$ containing the point $c'$ such that $C \cap U'=D_{p'} \cap U'$.
We can take a point $d \in C \cap U'$ with $(p',d) \in E'$ because $(p',c')   \in \overline{E'} \cap E_{\max}$.
Take an open box $V'$ in $M^m$ contained in $U'$ and containing the point $d$.
We obviously have $C \cap V'=D_{p'} \cap V'$.
It means that $(p',d) \in S'$.
The cell $E'$ is contained in $S'$ because the cell decomposition partitions the set $S'$.
We have $\dim E'>\dim E_{\max}$, $\dim \pi'(E')=\dim \pi'(E_{\max})=\dim P$ and $E' \subset S'$.
It is a contradiction to the definition of $E_{\max}$.
We have demonstrated that $E' \cap (P \times C)$ is an empty set.

We can take a point $(p_1,c_1) \in E_{\max}$ such that intersection $(P \times C) \cap (V_1 \times U_1)$ is contained in $E_{\max}$ for a sufficiently small open box $U_1$ in $M^m$ containing the point $c_1$ and a sufficiently small open box $V_1$ in $M^n$ containing the point $p_1$ by the previous claim and Lemma \ref{lem:ato}.
It means that $(C \times P) \cap (U_1 \times V_1)$ is contained in $S$.
Consider the restriction $g$ of $f$ to the set $(C \times P) \cap (U_1 \times V_1)$.
The set of points at which $g$ is discontinuous is $D \cap (U_1 \times V_1)$, and $g$ is discontinuous everywhere because $S$ is contained in $D$.
It contradicts to Proposition \ref{prop:pre}(7).
We have demonstrated that $\dim \pi(S) < \dim P$.

We next demonstrate that $\dim \pi(D)<\dim P$.
We lead to a contradiction assuming the contrary.
Set $T=D \setminus \pi^{-1}(\pi(S))$.
We have $\dim \pi(T)= \dim P$ by Proposition \ref{prop:pre}(5) because $\dim\pi(S)<\dim P$.
There exists a point $(c,p) \in T$ such that $\dim \pi(T \cap W)=\dim P$ for all open box $W$ in $M^{m+n}$ containing the point $(c,p)$ by Lemma \ref{lem:proj_dim}.
Fix an arbitrary $\varepsilon>0$.
Since $f$ is uniformly equi-continuous with respect to $P$ by the assumption and Proposition \ref{prop:equi-cont}, there exists $\delta>0$ satisfying the following condition:
\begin{equation}\label{eq:equi-cont1}
\forall q \in P, \ \forall x,x' \in C,\ |x-x'|<\delta \Rightarrow |f(x,q)-f(x',q)|<\varepsilon/3\text{.}
\end{equation}
Since $T \cap S=\emptyset$, there exists $c_1 \in C$ such that $|c-c_1|<\delta/2$ and $(c_1,p) \not\in D$.
There exists $\delta'>0$ such that 
\begin{equation}\label{eq:equi-cont2}
\forall q \in P,\ |q-p|<\delta' \Rightarrow |f(c_1,q)-f(c_1,p)|<\varepsilon/3
\end{equation}
because $f$ is continuous at $(c_1,p)$.

Consider an arbitrary point $(c',p') \in C \times P$ with $|c-c'|<\delta/2$ and $|p-p'|<\delta'$.
We have $|f(c_1,p)-f(c,p)|<\varepsilon/3$ by the inequality (\ref{eq:equi-cont1}) because $|c-c_1|<\delta/2$.
We also have $|f(c',p')-f(c_1,p')|<\varepsilon/3$ by (\ref{eq:equi-cont1}) because $|c'-c_1| \leq |c'-c|+|c-c_1|<\delta$.
We get 
\begin{align*}
|f(c',p')-f(c,p)| &\leq |f(c',p')-f(c_1,p')|+|f(c_1,p')-f(c_1,p)|+|f(c_1,p)-f(c,p)|\\
&<\varepsilon
\end{align*}
by the above inequalities together with the inequality (\ref{eq:equi-cont2}).
We have demonstrated that $f$ is continuous at $(c,p)$.
It is a contradiction to the condition that $(c,p) \in T \subset D$.
We have finished the proof of the theorem.
\end{proof}

\section{Definable Tietze extension theorem and o-minimality}\label{sec:tietze}

We treat the assertions satisfied in a DCULOAS structure in the previous sections.
We consider a slightly different type of problem in this section.
We consider whether a DCULOAS structure satisfying definable Tietze extension property is o-minimal or not.

\begin{definition}
A structure $\mathcal M=(M,\ldots)$ enjoys  \textit{definable Tietze extension property} if, for any positive integer $n$, any definable closed subset $A$ of $M^n$ and any continuous definable function $f:A \rightarrow M$, there exists a definable continuous extension $F:M^n \rightarrow M$ of $f$.
\end{definition}

We first prove the following lemma.

\begin{lemma}\label{lem:unbound0}
Consider a definably complete expansion of a densely linearly ordered abelian group  $\mathcal M=(M,<,+,0,\ldots)$.
If the structure $\mathcal M$ has a strictly monotone definable homeomorphism between a bounded open interval and an unbounded open interval, any two open intervals are definably homeomorphic and there exists a definable strictly increasing homeomorphism between them.
\end{lemma}
\begin{proof}
By the assumption, there exists a strictly monotone definable homeomorphism $\varphi:I \rightarrow J$, where $I$ is a bounded open interval and $J$ is an unbounded open interval.
We may assume that $I=]0,u[$ for some $u>0$.
In fact, an open interval $]u_1,u_2[$ is obviously definably homeomorphic to $]0,u_2-u_1[$.
We may further assume that $\varphi$ is strictly increasing because the map $\tau:]0,u[ \rightarrow ]0,u[$ defined by $\tau(t)=u-t$ is a definable homeomorphism.

We next reduce to the case in which $J=]0,\infty[$.
We have only three possibilities; that is $J=]v,+\infty[$, $J=]-\infty,v[$ and $J=M$ for some $v \in M$.
In the first and second cases, we may assume that $J=]0,\infty[$ because $J=]v,+\infty[$ and  $J=]-\infty,v[$ are obviously definable homeomorphic to $]0,\infty[$.
In the last case, set $u'=\varphi^{-1}(0)$.
Then the restriction of $\varphi$ to the open interval $]0,u'[$ is a definable homeomorphism between $]0,u'[$ and $]-\infty,0[$.
Hence, we can reduce to the second case.
We have constructed a strictly increasing definable homeomorphism $\varphi:]0,u[ \rightarrow ]0,\infty[$.
We fix such a homeomorphism.

We next construct a definable strictly increasing homeomorphism between an arbitrary  bounded open interval and $]0,\infty[$.
We may assume that the bounded interval is of the form $]0,v[$.
We have nothing to do when $v=u$.
When $v<u$, the map defined by $\varphi(t+u-v)-\varphi(u-v)$ for all $t \in ]0,v[$ is a definable homeomorphism between $]0,v[$ and $]0,\infty[$.
When $v>u$, consider the map $\psi:]0,v[ \rightarrow ]0,\infty[$ given by $\psi(t)=t$ for all $t \leq v-u$ and $\psi(t)=\varphi(t+u-v)+v-u$ for the other case.
It is a desired definable homeomorphism.
We have constructed a definable homeomorphism between $]0,u[$ and all open intervals other than $M$.

The remaining task is to construct a definable homeomorphism between $]0,u[$ and $M$.
There exists a strictly increasing definable homeommorphisms $\psi_1:]0,u/2[ \rightarrow ]-\infty,0[$ and $\psi_2:]u/2,u[ \rightarrow ]0,\infty[$.
The definable map $\psi:]0,u[ \rightarrow M$ given by $\psi(t)=\psi_1(t)$ for $t<u/2$, $\psi(t)=0$ for $t=u/2$ and $\psi(t)=\psi_2(t)$ for $t>u/2$ is a definable homeomorphism.
They are well-defined because $(M,+)$ is a divisible group by \cite[Proposition 2.2]{M}.
\end{proof}

\begin{lemma}\label{lem:unbound}
Consider a uniformly locally o-minimal expansion of the second kind of a densely linearly ordered abelian group  $\mathcal M=(M,<,+,0,\ldots)$.
If the structure $\mathcal M$ has a strictly monotone definable homeomorphism between a bounded open interval and an unbounded open interval, it is o-minimal.
\end{lemma}
\begin{proof}
Since the map $\tau:]0,u[ \rightarrow ]0,u[$ defined by $\tau(t)=u-t$ is a definable homeomorphism, we may assume that there exists a strictly decreasing definable homeomorphism $\varphi:]0,u[ \rightarrow ]0,+\infty[$ for some $u>0$ by Lemma \ref{lem:unbound0}.

Let $X$ be an arbitrary definable subset of $M$.
We show that it is a finite union of points and open intervals.

We first consider the case in which $X$ is bounded.
We may assume that $x>0$ for all $x \in X$ by shifting the definable set $X$ if necessary.
Take $N>0$ with $x<N$ for all $x \in X$.
Consider the map $\psi:]0, \infty[ \times ]0,u[ \rightarrow ]0,u[$ defined by $$\psi(x,y)=\varphi^{-1}(x+\varphi(y))\text{.}$$
Set $Z=\{(z,y) \in ]0, u[ \times ]0,u[\;|\;z=\psi(x,y) \text{ for some }x  \in X\}$.
The notation $Z_y$ denotes the set $\{z \in M\;|\; (z,y) \in Z\}$ for all $y \in M$.
Since $\mathcal M$ is a uniformly locally o-minimal structure of the second kind, there exists $c>0$ and $d>0$ such that, for any $0<y<d$, the intersection $Z_y \cap ]-c,c[$ is a finite union of points and open intervals.
We may assume that $c<u$ taking a smaller $c$ if necessary.
Take $0<y<c$.
For all $x \in X$, we have 
\begin{align*}
\psi(x,y) &=\varphi^{-1}(x+\varphi(y))<\varphi^{-1}(x+\varphi(c))<\varphi^{-1}(\varphi(c))=c
\end{align*}
because $x>0$ when $x \in X$.
It means that $\psi(X,y)$ is contained in the open interval $]0,c[$.
Fix a sufficiently small $y>0$ with $y< \min\{c,d\}$.
We have $\psi(X,y)=Z_y \cap ]-c,c[$, which is a finite union of points and open intervals.
Since the map $\psi(\cdot, y)$ is a definable homeomorphism for the fixed $y$, the set $X$ itself is a finite union of points and open intervals.

We next consider the case in which $X$ is unbounded.
Set $X_+=\{x \in X\;|\; x>0\}$ and $X_-=\{x \in X\;|\; x<0\}$.
Consider the sets $\varphi^{-1}(X_+)$ and $\varphi^{-1}(-X_-)$.
They are bounded definable subsets of $M$, and they are finite unions of points and open intervals.
Therefore, $X$ itself is a a finite union of points and open intervals because $\varphi$ is definable strictly decreasing homeomorphism.
We have demonstrated that $\mathcal M$ is o-minimal.
\end{proof}

\begin{definition}
Consider an expansion of a densely linearly ordered abelian group  $\mathcal M=(M,<,+,0,\ldots)$.
It is called \textit{archimedean} if, for any positive $a,b \in M$, there exists a positive integer $n$ with $na>b$.
Here, $na$ denotes the sum of $n$ copies of $a$.
\end{definition}

The following theorem is the last main theorem of this paper. 
Its proof is inspired by \cite[Example 3.4]{AT}.
\begin{theorem}\label{thm:tietze}
Consider an archimedean DCULOAS structure $\mathcal M =(M,<,+,0,\ldots)$.
If the structure $\mathcal M$ enjoys definable Tietze extension property, the structure $\mathcal M$ is o-minimal and it has a strictly monotone definable homeomorphism between a bounded open interval and an unbounded open interval.
\end{theorem}
\begin{proof}
We have only to construct a strictly monotone definable homeomorphism between a bounded open interval and an unbounded open interval by Lemma \ref{lem:unbound}.
Fix $c>0$.
Set $X=\{(x,y) \in M^2\;|\; x \leq 0 \text{ or } x \geq c\}$.
Consider the definable continuous map $f:X \rightarrow M$ given by $f(x)=-y$ if $x \leq 0$ and $f(x,y)=y$ otherwise.
Since the structure $\mathcal M$ enjoys definable Tietze extension property by the assumption, the function $f$ has a definable continuous extension $F:M^2 \rightarrow M$.

Fix $\varepsilon >0$.
The map $g_y:[0,c] \rightarrow M$ given by $g_y(x)=F(x,y)$ are uniformly continuous for all $y \in M$ by Corollary \ref{cor:uniform}.
Therefore there exists $\delta_y>0$ such that the condition $|x-x'|<\delta_y$ implies that $|F(x,y)-F(x',y)|<\varepsilon$ for all $x,x' \in [0,c]$.
It means that the definable function $\varphi:M_{>0} \rightarrow M_{>0}$ defined by 
$$
\varphi(y)=\sup\{0<\delta \leq c\;|\; \forall x,x' \in [0,c],\ |x-x'|<\delta \Rightarrow |F(x,y)-F(x',y)|<\varepsilon\}
$$
is well-defined.

The infimum $\inf \varphi(M_{>d})$ always exists for any $d>0$ because $\mathcal M$ is definably complete.
We prove that $$\inf \varphi(M_{>d})=0\text{.}$$
We lead to a contradiction assuming that $\inf \varphi(M_{>d})>0$ for some $d>0$.
Take a positive $\mu>0$ with $\mu<\inf \varphi(M_{>d})$.
We have $\varphi(y)>\mu$ for all $y>d$.
There exists a positive integer $n$ with $n\mu>c$ because $\mathcal M$ is archimedean.
Set $x_i=\frac{i}{n}c$ for all $0 \leq i \leq n$.
They are well-defined because $(M,+)$ is a divisible group by \cite[Proposition 2.2]{M}.
We have $|x_i-x_{i-1}|=\frac{c}{n} < \mu <\varphi(y)$ for all $y>d$ and $1 \leq i \leq n$.
For any $y>d$, we get 
$$
2y =|F(c,y)-F(0,y)| \leq \sum_{i=1}^n |F(x_{i},y)-F(x_{i-1},y)|< n\varepsilon
$$
by the definition of $\varphi(y)$.
It is a contradiction because $y$ is an arbitrary element with $y>d$.
We have demonstrated that $\inf \varphi(M_{>d})=0$.

Fix $d>0$.
Since $\varphi(M_{>d})$ is a set definable in a locally o-minimal structure and $\inf \varphi(M_{>d})=0$, there exists $u>0$ such that the open interval $]0,u[$ is contained in $\varphi(M_{>d})$.
Consider the definable function $\iota:]0,u[ \rightarrow M_{>0}$ given by 
$$
\iota(t)=\inf\{y \in M_{>d}\;|\;\varphi(y)=t\}\text{.}
$$
We define $\psi:]0,u[ \rightarrow M_{>0}$ as follows:
Set $\psi(t)=\iota(t)$ when $t=\varphi(\iota(t))$.
Otherwise, the set $$T_t=\{y \in M_{>d}\;|\; y >\iota(t),\  \forall y', \ \iota(t)<y'<y \Rightarrow \varphi(y')=t\}$$ is not empty because of local o-minimality.
The supremum $e(t) = \sup T_t \in M \cup \{+\infty\}$ exists by definable completeness.
Set $\psi(t)=\dfrac{\iota(t)+e(t)}{2}$ when $e(t)<\infty$, and set $\psi(t)=\iota(t)+c$ otherwise.
We have $\varphi(\psi(t))=t$ by the definition.

For any $0<u'<u$, the restriction of $\psi$ to the open interval $]0,u'[$ is unbounded.
Assume the contrary.
There exists $0<u'<u$ and $v>0$ such that $\psi(]0,u'[)$ is contained in $[d,v]$.
Since the closed box $[0,c] \times [d,v]$ is bounded, there exists $\widetilde{\delta}>0$ such that the following condition holds true by Corollary \ref{cor:uniform}:
$$
\forall (x,y),(x',y') \in [0,c] \times [d,v],\ |(x,y)-(x',y')|<\widetilde{\delta} \Rightarrow |F(x,y)-F(x',y')|<\varepsilon\text{.}
$$
It implies that $\varphi(y) \geq \widetilde{\delta}$ for all $d \leq y \leq v$.
We may assume that $\widetilde{\delta}<u'$ taking a smaller $\widetilde{\delta}$ if necessary.
Take $t>0$ smaller than $\widetilde{\delta}$.
We have $d \leq \psi(t) \leq v$ and $\varphi(\psi(t)) = t < \widetilde{\delta}$.
Contradiction.
We have proven that the restriction of $\psi$ to the open interval $]0,u'[$ is unbounded for any $0<u'<u$.

Taking a smaller $u>0$ if necessary, we may assume that the function $\psi$ is continuous and monotone by Theorem \ref{thm:mono} and Lemma \ref{lem:dc_mono}.
Since the restriction of $\psi$ to the open interval $]0,u'[$ is unbounded for any $0<u'<u$, it is strictly decreasing.
The restriction $\psi$ to the open interval $]0,u/2[$ is a strictly monotone definable homeomorphism between the bounded open interval $]0,u/2[$ and the unbounded open interval $]\psi(u/2),\infty[$.
\end{proof}

We have only proved that an archimedean DCULOAS structure which enjoys definable Tietze extension property is o-minimal in Theorem \ref{thm:tietze}.
The following conjecture is still open.
\begin{conjecture}
A DCULOAS structure is o-minimal when it enjoys definable Tietze extension property.
\end{conjecture}

\begin{remark}
Consider a definably complete locally o-minimal structure $\mathcal M=(M,<,\ldots)$ such that the projection image of any definable discrete set under a coordinate projection.
We can define the dimension of set definable in this structure without using cells.
The assertions in Section \ref{sec:review} hold true for this structure except Lemma \ref{lem:ato}.
Lemma \ref{lem:ato} is only used for the proof of Theorem \ref{thm:discont}.
The assertions in Section \ref{sec:property} and Section \ref{sec:tietze} except it hold true if the structure also satisfies the definable choice lemma given in Lemma \ref{lem:definable_choice}.
\end{remark}

\section{Appendix: Tietze extension property in o-minimal structures}\label{sec:appendix}

We study o-minimal structures in this section.
Readers who are not familiar with o-minimal structures should consult \cite{vdD, KPS, PS}.
We give an equivalent condition for an o-minimal expansion of an ordered group enjoying definable Tietze extension property.
We first introduce two lemmas.

\begin{lemma}\label{lem:unbound1}
Consider an o-minimal structure.
The structure has a definable beijection between a bounded interval and an unbounded interval if and only if it has a definable homeomorphism between a bounded interval and an unbounded interval.
\end{lemma}
\begin{proof}
It is immediate from the monotonicity theorem \cite[Chapter 3, Theorem 1.2]{vdD}.
\end{proof}

\begin{lemma}\label{lem:partition}
Consider an o-minimal structure $\mathcal M=(M,<,\ldots)$ and a definable function $f:[a,b] \times M \rightarrow M$, where $a$ and $b$ are constants.
The definable function $f_x:M \rightarrow M$ is given by $f_x(y)=f(x,y)$ for any $a \leq x \leq b$.
Let $\pi:M^2 \rightarrow M$ denote the projection onto the first coordinate.

Then there exists a partition of $[a,b] \times M$ into cells $\{C_1, \ldots, C_n\}$ such that, for any cell $C_i$, the cell $C_i$ does not contain the set of the form $\{c\} \times M$ with $c \in M$ and one of the following three conditions is satisfied:
\begin{itemize}
\item The restriction $f_x|_{(C_i)_x}$ of $f_x$ to $(C_i)_x$ is strictly increasing for any $x \in \pi(C_i)$;
\item The restriction $f_x|_{(C_i)_x}$ is strictly decreasing for any $x \in \pi(C_i)$;
\item The restriction $f_x|_{(C_i)_x}$ is constant for any $x \in \pi(C_i)$.
\end{itemize}
Here, the notation $(C_i)_x$ denotes the fiber $\{y \in M\;|\;(x,y) \in C_i\}$.
\end{lemma}
\begin{proof}
Consider the following definable sets:
\begin{align*}
X_{\text{inc}} &= \{(x,y) \in [a,b] \times M\;|\; \exists y_1, \exists y_2,\ (y_1<y<y_2)\\
&\qquad  \wedge (f_x \text{ is strictly increasing on }]y_1,y_2[)\}\\
X_{\text{dec}} &= \{(x,y) \in [a,b] \times M\;|\; \exists y_1, \exists y_2,\ (y_1<y<y_2)\\
&\qquad  \wedge (f_x \text{ is strictly decreasing on }]y_1,y_2[)\}\\
X_{\text{con}} &= \{(x,y) \in [a,b] \times M\;|\; \exists y_1, \exists y_2,\ (y_1<y<y_2)\\&\qquad\wedge (f_x \text{ is constant on }]y_1,y_2[)\}
\end{align*}
The fiber $(X_{\text{inc}} \cup X_{\text{dec}} \cup X_{\text{con}})_x$ at $x$ is dense in $M$ for any $a \leq x \leq b$ by the monotonicity theorem.
Apply the definable cell decomposition theorem \cite[Chapter 3, Theorem 2.12]{vdD}. We can get a cell decomposition of $M^2$ partitioning $X_{\text{inc}}$, $X_{\text{dec}}$,  $X_{\text{con}}$, $[a,b] \times M$ and $[a,b] \times \{a\}$.
The cells contained in $[a,b] \times M$ satisfy the requirements of the lemma.
\end{proof}

The following proposition is a part of \cite[Fact 1.6]{E}.
\begin{proposition}\label{prop:edmundo}
Consider an o-minimal expansion of an ordered group $\mathcal M=(M,<,+,0,\ldots)$.
The followings are equivalent:
\begin{enumerate}
\item[(1)] There exists a definable bijection between a bounded interval and an unbounded interval.
\item[(2)] In $\mathcal M$, we can define a real closed field whose universe is an unbounded subinterval of $M$ and whose ordering agrees with $<$.
\end{enumerate}
\end{proposition}

Finally, we prove the following theorem.
\begin{theorem}\label{thm:appendix}
Consider an o-minimal expansion $\mathcal M=(M,<,+,0,\ldots)$ of an ordered group.
The followings are equivalent:
\begin{enumerate}
\item[(1)] There exists a definable bijection between a bounded interval and an unbounded interval. 
\item[(2)] The structure $\mathcal M$ enjoys definable Tietze extension property.
\end{enumerate}
\end{theorem}
\begin{proof}
We first demonstrate that the condition (1) implies the condition (2).

There exist an unbounded subinterval $I$ of $M$, two elements $0^*$ and $1^*$ in $I$, and definable functions $\oplus, \otimes: I \times I \rightarrow I$ such that $(I,0^*,1^*,\oplus,\otimes)$ is a real closed field with the ordering $<$ by Proposition \ref{prop:edmundo}.
The subinterval $I$ is obviously an open interval.

If $I=M$, the assertion (2) directly follows from the original definable Tietze extension theorem \cite[Chapter 8, Corollary 3.10]{vdD}.

We next consider the other case.
Consider a definable continuous function $f:A \rightarrow M$ defined on a definable closed subset $A$ of $M^n$.
We construct a definable continuous extension $F:M^n \rightarrow M$ of the function $f$.
There exists a definable homeomorphism $\sigma:M \rightarrow I$ by Lemma \ref{lem:unbound1} and Lemma \ref{lem:unbound0}.
The notation $\sigma_n$ denotes the homeomorphism from $M^n$ onto $I^n$ induced by $\sigma$.
The definable set $\sigma_n(A)$ is contained $I^n$.
Consider the definable continuous function $f_{\sigma}:\sigma_n(A) \rightarrow I$ defined by $f_{\sigma}(x)=(\sigma \circ f \circ \sigma_n^{-1})(x)$.
Its graph is obviously contained in $I^{n+1}$.

We consider a new structure $\mathcal I$ whose universe is $I$.
Let $\mathfrak S_n$ be the set of all subset of $I^n$ definable in $\mathcal M$.
Set $\mathfrak S = \bigcup_{n \geq 0} \mathfrak S_n$.
For any $S \in \mathfrak S$, we introduce new predicate symbol $R_S$ and we define $\mathcal I \models R_S(x)$ by $x \in S$. 
The structure $\mathcal I=(I, <, \{R_S\}_{S \in \mathfrak S})$ is obviously an o-minimal structure.
Since the operators $\oplus$ and $\otimes$ are definable in $\mathcal I$, the structure $\mathcal I$ is an o-minimal expansion of an ordered field.
The $\mathcal M$-definable set $\sigma_n(A)$ and the $\mathcal M$-definable function $f_{\sigma}$ are also definable in the structure $\mathcal I$.
Note that the function $f_{\sigma}$ is also continuous under the topology induced by the ordering of the real closed field $(I,0^*,1^*,\oplus,\otimes)$ because the both structure $\mathcal I$ and $\mathcal M$ shares the same order $<$.
There exists a continuous extension $F_\sigma:I^n \rightarrow I$ of $f_\sigma$ definable in $\mathcal I$ by the original definable Tietze extension theorem \cite[Chapter 8, Corollary 3.10]{vdD}.
The function $F_\sigma$ is also definable in $\mathcal M$ by the definition of the structure $\mathcal I$.
Set $F=\sigma^{-1} \circ F_{\sigma} \circ \sigma_n$.
It is the desired definable continuous extension of $f$ definable in $\mathcal M$.

We next show that the condition (2) implies the condition (1).
We construct a definable bijection between a bounded interval and an unbounded interval.
Take a positive element $c$ in $M$. 
Consider the definable closed set $A=\{(x,y) \in M^2\;|\; x \leq 0 \text{ or } x \geq c\}$ and the definable continuous function $f:A \rightarrow M$ given by $f(x,y)=y$ if $x \geq c$ and $f(x,y)=0$ otherwise.
By the condition (2), there exists a definable continuous extension $F:M^2 \rightarrow M$ of $f$.
The notation $g$ denotes the restriction of $F$ to $[0,c] \times M$.

Consider the sets $S_{t,y}=\{x \in [0,c]\;|\;g(x,y)=t\}$ for all $t \geq 0$ and $y \geq 0$.
The definable sets $S_{t,y}$ is not empty for $y>t$ by the intermediate value theorem \cite[Corollary 1.5]{M}.
The definable function $\varphi_t:M_{>t} \rightarrow [0,c]$ is given by $\varphi_t(y)=\sup S_{t,y}$.
For any $t>0$, there exists a nonnegative $u_t$ such that the restriction $\varphi_t|_{M_{>u_t}}$ of the function $\varphi_t(y)$ to $M_{>u_t}=\{y \in M\;|\; y>u_t\}$ is continuous and strictly monotone or constant for $y>u_t$ by the monotonicity theorem.

We consider the following two cases separately.
\begin{enumerate}
\item[(a)] The restriction $\varphi_t|_{M_{>u_t}}$ is continuous and strictly monotone for some $t>0$.
\item[(b)] The restriction $\varphi_t|_{M_{>u_t}}$ is constant for any $t>0$.
\end{enumerate}
In the case (a), the restriction $\varphi_t|_{M_{>u_t}}$ gives a bijection between a bounded interval and an unbounded interval.
We have finished the proof in this case.
We concentrate on the case (b).
By the definition of the function $\varphi_t$, the following assertion holds true:
\begin{quotation}
For any $t>0$, there exist a point $x_t \in [0,c]$ and a nonnegative $u_t$ such that $g(x_t,y)=t$ for all $y>u_t$.
\end{quotation}
In fact, we have only to take $y'>u_t$ and set $x_t=\varphi_t(y')$. 

We introduce terminologies only used in the proof.
A cell $C$ in $M^2$ is called a \textit{roof} if $C$ is one of the following forms:
\begin{enumerate}
\item[(i)] $C=\{(a,y) \in M^2\;|\; y>b\}$;
\item[(ii)] $C=\{(x,y) \in M^2\;|\; a_1<x<a_2, \ y> \psi(x)\}$.
\end{enumerate}
Here, $a,b, a_1,a_2$ are constants and $\psi$ is a definable continuous function defined on the open interval $]a_1,a_2[$.
The first one is called a \textit{non-open roof} and the second one is called an \textit{open roof}.

Apply Lemma \ref{lem:partition} to the definable continuous function $g$.
Let $\{C_1, \ldots, C_n\}$ be the obtained partition of $[0,c] \times M$.
Permuting if necessary, we may assume that $C_i$ are open roofs for $1 \leq i \leq k$ and $C_i$ are non-open roofs for $k < i \leq l$.
Consider the sets $$T_i = \{t>0\;|\; \exists y,\ y>u_t \text{ and } (x_t,y) \in C_i\}$$ for all $1 \leq i \leq k$.
We demonstrate that the set $T_m$ is unbounded for some $1 \leq m \leq k$.
Assume the contrary.
There exists $v$ such that, for any $t>v$, the intersections of $V_t=\{(x_t,y) \in M^2\;|\; y>u_t\}$ with $C_i$ are empty sets for all $1 \leq i \leq k$.
However, by the definition of the roofs, $V_t$ intersects with some roof.
Therefore, $V_t$ intersects with some non-open roof for any $t>v$.
Since we have $x_t \neq x_{t'}$ when $t \neq t'$, a non-open roof can intersect with only one of $V_t$'s.
It means that the set $\{t \in M\;|\; t>v\}$ is finite.
It is a contradiction.
We have shown that the set $T_m$ is unbounded for some $1 \leq m \leq k$.
We fix such $m$.

The open roof $C_m$ is of the form:
$$C_m=\{(x,y) \in M^2\;|\; a_1<x<a_2, \ y> \psi(x)\}\text{,}$$
where $a_1$ and $a_2$ are constants and $\psi$ is a definable continuous function defined on $]a_1,a_2[$.
We demonstrate that the function $\psi$ is unbounded.
Once $\psi$ is proved to be unbounded, the restriction of $\psi$ to some bounded interval is a definable homeomorphism between the bounded interval and an unbounded interval by the monotonicity theorem.
We have demonstrated the theorem.

The remaining task is to demonstrate that $\psi$ is unbounded.
By the definition of $C_m$, the functions $g_x: M_{> \psi(x)} \rightarrow M$ given by $g_x(y)=g(x,y)$ satisfies one of the following conditions:
\begin{itemize}
\item The function $g_x$ are constant for all $a_1<x<a_2$;
\item The function $g_x$ are strictly monotone for all $a_1<x<a_2$.
\end{itemize}
Take $t \in T_m$, then $g_{x_t}(y)$ is constant on $y> \max\{u_t,\psi(x_t)\}$ by the assumption (b).
Therefore, the first condition is satisfied.
We lead to a contradiction assuming that the function $\psi$ is bounded.
Take $w \in M$ with $w>\psi(x)$ for all $a_1 < x <a_2$.
We can take the maximum $N$ of the restriction of $g$ to $[0,c] \times \{w\}$ by the max-min theorem \cite[Corollary of Proposition 1.10]{M}.
Since $g_x$ are constant for all $a_1<x<a_2$, we have $g(x,y) \leq N$ for all $(x,y) \in C_m$.
Fix $t \in M$ with $t>N$ and $t \in T_m$.
Such $t$ exists because $T_m$ is unbounded.
We can take $y_0 \in M$ with $(x_t,y_0) \in C_m$ by the definition of $T_m$.
We have $g(x_t,y_0)=t>N$.
Contradiction to the definition of $N$. 
\end{proof}

\begin{remark}
We make a comment on Theorem \ref{thm:appendix}.
Miller and Starchenko studied asymptotic behavior of o-minimal expansion of an ordered group $\mathcal M=(M,<,+, \ldots)$ in \cite{MS}. 
They introduced the notion of linear boundedness.
An o-minimal structure is called \textit{linearly bounded} if, for any definable function $f:M \rightarrow M$, there exists a definable automorphism $\lambda:M \rightarrow M$ with $|f(x)| \leq \lambda(x)$ for all sufficiently large $x \in M$.
Their main theorem is that there exists a definable binary operation $\cdot$  such that $(M,<,+,\cdot)$ is an ordered real closed field when the structure is not linearly bounded.

Peterzil and Edmundo studied the subclass of linearly bounded o-minimal expansions of ordered groups \cite{E,P}.
An o-minimal structure $\mathcal M$ is \textit{semi-bounded} if any set definable in $\mathcal M$ is already definable in the o-minimal structure generated by the collection of all bounded sets definable in $\mathcal M$.
Edmundo gave equivalent conditions for an o-minimal expansion of an ordered group to be semi-bounded in \cite[Fact 1.6]{E}.
The condition (1) in our theorem is the negation of one of them.
Theorem \ref{thm:appendix} gives a new equivalent condition.
An o-minimal expansion of an ordered group is semi-bounded if and only if it does not have definable Tietze extension property.
\end{remark}

\end{document}